\documentclass[10pt,a4]{amsart}
\usepackage[dvips]{graphicx}
\usepackage{amsmath,amssymb}
\usepackage{txfonts}
\usepackage[all]{xy}
\makeatletter
 
  \@addtoreset{equation}{section}
 \makeatother
\newtheorem{thm}{Theorem}[section]
\newtheorem{prp}[thm]{Proposition}
\newtheorem{lem}[thm]{Lemma}

\newcommand{ \bm}[1]{\boldsymbol{#1}}

\newcommand{ \Map}{{\rm Map}\,}

\newcommand{ \SU}{S\!U}

\begin{document}
\title{A note on homotopy types of connected components of $\Map(S^4,B\SU(2))$}
\author{Mitsunobu Tsutaya}
\thanks{E-mail address: tsutaya@math.kyoto-u.ac.jp}
\date{}
\maketitle

\section{Introduction}
By \cite{Got72}, connected components of $\Map(S^4,B\SU(2))$ is the classifying spaces of gauge groups of principal $\SU(2)$-bundles over $S^4$.
Tsukuda \cite{Tsu01} has investigated the homotopy types of connected components of $\Map(S^4,B\SU(2))$.
But unfortunately, the proof of Lemma 2.4 in \cite{Tsu01} is not correct for $p=2$.
In this paper, we give a complete proof.
Moreover, we investigate the further divisibility of $\epsilon_i$ defined in \cite{Tsu01}.
In \cite{Tsu}, it is shown that divisibility of $\epsilon_i$ have some information about $A_i$-equivalence types of the gauge groups.
\par
In \S 2, we review the definition of $\epsilon_i$ and the motivation in homotopy theory.
In \S 3, 4, 5 and 6, we investigate the divisibility of $\epsilon_i$.
These four sections are purely algebraic.
In \S 7, we apply these results to $A_n$-types of gauge groups.
Especially, we estimate the growth of the number of $A_n$-types of gauge groups of principal $\SU(2)$-bundles over $S^4$.
\par
The author is so grateful to Professor Akira Kono and Doctors Kentaro Mitsui and Minoru Hirose for fruitful discussions.

\section{Definition and motivation}
We review the definition of $\{\epsilon_i\}$.
Let $P_k$ be a principal $\SU(2)$-bundle over $S^4$ with $\langle c_2(P_k),[S^4]\rangle =k\in \bm{Z}$.
According to \cite{Tsu}, the gauge group $\mathcal{G}(P_k)$ is $A_n$-equivalent to a topological group $\mathcal{G}(P_0)=\Map(S^4,\SU(2))$ if and only if the map
\[
S^4\vee \bm{H}P^n\stackrel{k\vee i}{\rightarrow} \bm{H}P^\infty \vee \bm{H}P^\infty \stackrel{\nabla}{\rightarrow}\bm{H}P^\infty
\]
extends over $S^4\times \bm{H}P^n$, where $k:S^4\to \bm{H}P^\infty$ is a classifying map of $P_k$, $i:\bm{H}P^n\to \bm{H}P^\infty$ is the inclusion and $\nabla:\bm{H}P^\infty \vee \bm{H}P^\infty \to \bm{H}P^\infty$ is the folding map.
\par
Now, we assume there exists the following homotopy commutative diagram:
\[
\xymatrix{
S^4\vee \bm{H}P^n \ar[r]^-{k\vee i} \ar[d]_j & \bm{H}P^\infty \vee \bm{H}P^\infty \ar[r]^-{\nabla} & \bm{H}P^\infty \ar[d]^{\rm localization}\\
S^4\times \bm{H}P^n \ar[rr]^-f & & \bm{H}P^\infty_{(p)} 
}
\]
where $p$ is a prime and $j:S^4\vee \bm{H}P^n\to S^4\times \bm{H}P^n$ is the inclusion.
\par
We denote the localization of the ring of integers by the prime ideal $(p)\subset \bm{Z}$ by $\bm{Z}_{(p)}$.
The $p$-localized complex $K$-theory $K(\bm{H}P^\infty_{(p)})_{(p)}$ of $\bm{H}P^\infty_{(p)}$ is computed as
\begin{align*}
K(\bm{H}P^\infty_{(p)})_{(p)}=\bm{Z}_{(p)}[a].
\end{align*}
We may assume that there exists the generator $b\in H^4(\bm{H}P^\infty_{(p)};\bm{Q})$ such that
\begin{align*}
ch\,a=\sum_{j=1}^\infty\frac{2b^j}{(2j)!}.
\end{align*}
Similarly, take $u\in \tilde K(S^4)_{(p)}$ and $s\in H^4(S^4;\bm{Q})$ such that $ch\, u=s$.
Then, $f^*b=ks\times 1+1\times b$ in $H^4(S^4 \times \bm{H}P^n;\bm{Q})$ and
\begin{align*}
f^*a=ku\times 1+1\times a+\sum_{i=1}^n\epsilon _i(k)u\times a^i
\end{align*}
in $\tilde K(S^4 \times \bm{H}P^n)_{(p)}$, where $\epsilon _i(k) \in \bm{Z}_{(p)}$.
We calculate $f^*ch\, a$ and $ch\, f^*a$ as follows:
\begin{align*}
f^*ch\,a=f^*\sum_{j=1}^\infty\frac{2b^j}{(2j)!}=\sum_{j=1}^\infty \frac{2}{(2j)!}(ks\times 1+1\times b)^j=ks\times 1+\sum_{j=1}^n \left( \frac{k}{(2j+1)!}s\times b^j+\frac{2}{(2j)!}1\times b^j\right),
\end{align*}
\begin{align*}
ch\, f^*a=ch\,\left( ku\times 1+1\times a+\sum_{i=1}^\infty \epsilon _i(k)u\times a^i \right)
=&ks\times 1+1\times \sum_{j=1}^n\frac{2}{(2j)!}b^j+\sum_{i=1}^n\sum_{j=1}^n\epsilon _i(k)s\times \left( \sum_{j=1}^n\frac{2}{(2j)!}b^j \right)^i\\
=&ks\times 1+\sum_{j=1}^n\frac{2}{(2j)!}1\times b^j+\sum_{i=1}^n\sum_{l=1}^n\sum_{j_1+\cdots +j_i=l}\frac{2^i\epsilon _i(k)}{(2j_1)!\cdots (2j_i)!}s\times b^l.
\end{align*}
Then we have the following formula:
\begin{align*}
\frac{k}{(2l+1)!}=\sum_{i=1}^l\sum_{\substack{j_1+\cdots +j_i=l\\ j_1, \cdots, j_i\geq 1}}\frac{2^i\epsilon _i(k)}{(2j_1)!\cdots (2j_i)!}.
\end{align*}
From this formula, one can see that there exists the number $\epsilon _i\in \bm{Q}$ such that $\epsilon _i(k)=\epsilon _ik$ for each $i$.
Of course, the sequence $\{\epsilon_i\}_{i=1}^\infty$ satisfy the following formula for each $l$:
\begin{align*}
\frac{1}{(2l+1)!}=\sum_{i=1}^l\sum_{\substack{j_1+\cdots +j_i=l\\ j_1, \cdots, j_i\geq 1}}\frac{2^i\epsilon _i}{(2j_1)!\cdots (2j_i)!}.
\end{align*}
For example, $\epsilon _1=1/6$, $\epsilon _2=-1/180$, $\epsilon _3=1/1512$ etc.
From the above argument, if the map $({\rm localization})\nabla(k\vee i):S^4\vee \bm{H}P^n\to \bm{H}P^\infty_{(p)}$ extends over $S^4\times \bm{H}P^n$, then $\epsilon_1k,\cdots ,\epsilon_nk\in \bm{Z}_{(p)}$.
\par
Tsukuda \cite{Tsu01} defines a non-negative integer $d_p(k)$ for a prime $p$ and an integer $k$ as the largest $n$ such that there exists an extension of
\[
S^4\vee \bm{H}P^n\stackrel{k\vee i}{\rightarrow} \bm{H}P^\infty \vee \bm{H}P^\infty \stackrel{\nabla}{\rightarrow}\bm{H}P^\infty \stackrel{\rm localization}{\to}\bm{H}P^\infty_{(p)}
\]
over $S^4\times \bm{H}P^n$.
Remark $d_p(0)=\infty$.
Clearly, if we define $\epsilon_0=1$, then
\[
d_p(k)\leq d'_p(k):=\max \{\,n\in \bm{Z}_{\geq 0}\, |\, \epsilon_nk\in \bm{Z}_{(p)}\, \}.
\]
It is shown that $d_p(k)=d_p(k')$ for any prime $p$ if the classifying spaces $B\mathcal{G}(P_k)$ and $B\mathcal{G}(P_{k'})$ are homotopy equivalent.
Lemma 2.4 in \cite{Tsu01} asserts that $d'_p(k)<\infty$ (therefore $d_p(k)<\infty$) for $k\not=0$ and any prime $p$.
But the proof is invalid for $p=2$.
We will give a correct proof for this case in \S 4.
\par
We also review the result of \cite{Tsu}.
If $\mathcal{G}(P_k)$ and $\mathcal{G}(P_{k'})$ are $A_n$-equivalent, then $\min \{n,d_p(k)\}=\min \{n,d_p(k')\}$ for any prime $p$.
Let $p$ be an odd prime.
For $i<(p-1)/2$, $\epsilon_i\in \bm{Z}_{(p)}$.
For $(p-1)/2\leq i<p-1$, $p\epsilon_i\in \bm{Z}_{(p)}$.
Moreover, $\epsilon_{(p-1)/2}\not\in \bm{Z}_{(p)}$, $p\epsilon_{p-1}\not\in \bm{Z}_{(p)}$ and $p^2\epsilon_{p-1}\in \bm{Z}_{(p)}$.
We will generalize these results in \S 4 and 5.

\section{An explicit formula for $\epsilon_i$}
Algebraically, the sequence $\{\epsilon_i\}_{i=0}^\infty$ of rational numbers is defined by the following formula inductively:
\[
\frac{1}{(2l+1)!}=\sum_{i=1}^l\sum_{\substack{j_1+\cdots +j_i=l\\ j_1, \cdots, j_i\geq 1}}\frac{2^i\epsilon _i}{(2j_1)!\cdots (2j_i)!}
\]
and $\epsilon_0=1$.
Equivalently, $\{\epsilon_i\}$ is defined by the equality
\[
\sum_{l=0}^\infty \frac{x^l}{(2l+1)!}=\sum_{i=0}^\infty \epsilon _i\left( \sum_{j=1}^\infty \frac{2x^j}{(2j)!}\right)^i
\]
in the ring of formal power series $\bm{Q}[[x]]$.
\begin{prp}
The rational number $\epsilon_i$ is the $i$-th coefficient of the Taylor expansion of $1/f'(x)$ at $0\in \bm{C}$ for
\[
f(x)=\left(\cosh^{-1}\left(1+\frac{x}{2}\right)\right)^2,
\]
where $f$ is holomorphic in a neighborhood of $0$.
\end{prp}
\begin{proof}
Define a holomorphic function $h$ by
\[
h(x)=2\cosh \sqrt{x}-2=\sum_{i=1}^\infty \frac{2}{(2i)!}x^i
\]
in a neighborhood of 0.
Then $f$ given by the above formula is the inverse function of $h$.
We also define $g$ by
\[
g(x)=\sum_{i=0}^\infty \epsilon_ix^i.
\]
Then, formally, $h'(x)=g(h(x))$ by the definition of $\epsilon_i$.
Therefore, we have $g(x)=1/f'(x)$.
\end{proof}
The next proposition is proved by easy computation.
\begin{prp}
The holomorphic function $f$ satisfies the following differential equation:
\[
x(x+4)f''(x)+(x+2)f'(x)-2=0.
\]
\end{prp}
If the power series
\[
\sum_{i=0}^\infty a_ix^i
\]
satisfies the above equation, then
\[
a_1=1 \, , \, a_{i+1}=-\frac{i^2}{(2i+2)(2i+1)}a_i\hspace{1em}(i\geq 1).
\]
From these equations,
\[
a_i=(-1)^{i-1}\frac{2((i-1)!)^2}{(2i)!}
\]
for $i\geq 1$.
Hence,
\[
f(x)=\sum_{i=1}^\infty (-1)^{i-1}\frac{2((i-1)!)^2}{(2i)!}x^i
\]
and
\[
f'(x)=\sum_{i=0}^\infty (-1)^i\frac{(i!)^2}{(2i+1)!}x^i.
\]
Therefore,
\[
g(x)=\frac{1}{f'(x)}
=\sum_{j=0}^\infty(-1)^j\left(\sum_{i=1}^\infty (-1)^i\frac{(i!)^2}{(2i+1)!}x^i\right)^j
=1+\sum_{j=1}^\infty\sum_{i_1,\cdots ,i_j\geq 1}(-1)^{j+i_1+\cdots +i_j}\frac{(i_1!)^2\cdots (i_j!)^2}{(2i_1+1)!\cdots (2i_j+1)!}x^{i_1+\cdots +i_j}.
\]
This implies the following formula.
\begin{thm}\label{formula}
\[
\epsilon_l=\sum_{j=1}^l\sum_{\substack{i_1+\cdots +i_j=l\\ i_1, \cdots, i_j\geq 1}}(-1)^{j+l}\frac{(i_1!)^2\cdots (i_j!)^2}{(2i_1+1)!\cdots (2i_j+1)!}
\]
\end{thm}
\section{Divisibility of $\epsilon_i$ by 2}
For a prime $p$ and a rational number $n$, we denote the $p$-adic valuation of $n$ by $v_p(n)$.
Equivalently, if
\[
n=\frac{p^at}{p^bs}
\]
where $s$ and $t$ are integers prime to $p$, then $v_p(n)=a-b$.
First, we observe the divisibility of factorials.
\begin{lem}\label{factorial}
Let $p$ be a prime.
Then, for a integer
\[
n=n_rp^r+n_{r-1}p^{r-1}+\cdots +n_0
\]
where $0\leq n_i<p$ for each $i$,
\[
v_p(n!)=\frac{1}{p-1}(n-n_0-\cdots -n_r)
\]
\end{lem}
\begin{proof}
First, we remark the following:
\[
\# \{\, k\in \bm{Z}\,|\,1\leq k\leq n,k\, {\rm is\, divisible\, by\,}p^i\,\}=n_rp^{r-i}+n_{r-1}p^{r-i-1}+\cdots +n_i
\]
for $1\leq i\leq r$.
Hence
\begin{align*}
v_p(n!)=&(n_rp^{r-1}+n_{r-1}p^{r-2}\cdots +n_1)+(n_rp^{r-2}+n_{r-2}p^{r-3}+\cdots +n_2)+\cdots +n_r\\
=&n_r\frac{p^r-1}{p-1}+n_{r-1}\frac{p^{r-1}-1}{p-1}+\cdots +n_1=\frac{1}{p-1}(n-n_0-\cdots -n_r).
\end{align*}
\end{proof}
\par
For $p=2$, $v_2(n!)=n-n_0-\cdots -n_r$.
\begin{lem}\label{coeff2}
For a integer
\[
n=n_r2^r+n_{r-1}2^{r-1}+\cdots +n_0
\]
where $0\leq n_i<2$ for each $i$,
\[
v_2\left(\frac{(n!)^2}{(2n+1)!}\right) =-n_0-\cdots -n_r.
\]
\end{lem}
\begin{proof}
Since $v_2((2n+1)!)=2n-n_0-\cdots -n_r$ and $v_2(n!)=n-n_0-\cdots -n_r$, the formula above follows.
\end{proof}
Now, we observe the divisibility of $\epsilon_i$ by 2.
\begin{prp}
For positive integers $n_1,\cdots ,n_m$ and $l=n_1+\cdots +n_m$, if $n_j\geq 2$ for some $j$, then
\[
2^{l-1}\frac{(n_1!)^2\cdots (n_m!)^2}{(2n_1+1)!\cdots (2n_m+1)!}\in \bm{Z}_{(2)}.
\]
\end{prp}
\begin{proof}
From Lemma \ref{coeff2},
\begin{align*}
2^n\frac{(n!)^2}{(2n+1)!} \in \bm{Z}_{(2)}.
\end{align*}
Moreover, if $n>1$,
\begin{align*}
2^{n-1}\frac{(n!)^2}{(2n+1)!} \in \bm{Z}_{(2)}.
\end{align*}
The conclusion follows form this.
\end{proof}
From this proposition and Theorem \ref{formula},
\[
\epsilon_l\equiv 6^{-l} \mod \, 2^{-l+1}\bm{Z}_{(2)}.
\]
Then we have the following theorem.
\begin{thm}
\[
v_2(\epsilon_l)=-l
\]
\end{thm}
Then $d'_2(k)=v_2(k)$ (see \S 2 for the definition of $d'_2$) and Lemma 2.4 in \cite{Tsu01} for $p=2$ is proved.
\section{Divisibility of $\epsilon_i$ by an odd prime}
In general, divisibility of $\epsilon_i$ by an odd prime $p$ is more complicated than by 2 because the interval between a multiple of $p$ and the next one is longer.
But for $p=3$, we will have a similar result.
\begin{lem}\label{factorial2}
Let $p$ be a prime.
Then, for a integer
\[
n=n_rp^r+n_{r-1}p^{r-1}+\cdots +n_0
\]
where $0\leq n_i<p$ for each $i$,
\[
v_p((2n+1)!)\leq \frac{2}{p-1}(n-n_0-\cdots -n_r)+r+1
\]
\end{lem}
\begin{proof}
First, we remark the following:
\[
\# \{\, k\in \bm{Z}\,|\,1\leq k\leq 2n+1,k\, {\rm is\, divisible\, by\,}p^i\,\}\leq 2(n_rp^{r-i}+n_{r-1}p^{r-i-1}+\cdots +n_i)+1
\]
for $1\leq i\leq r+1$.
Hence
\begin{align*}
v_p((2n+1)!)\leq &2(n_rp^{r-1}+n_{r-1}p^{r-2}\cdots +n_1)+1+2(n_rp^{r-2}+n_{r-2}p^{r-3}+\cdots +n_2)+1+\cdots +2n_r+1+1\\
=&\frac{2}{p-1}(n-n_0-\cdots -n_r)+r+1.
\end{align*}
\end{proof}
\begin{lem}\label{coeffp}
For an odd prime $p$ and a positive integer $n$,
\[
v_p\left( \frac{(n!)^2}{(2n+1)!}\right)\geq -\frac{2n}{p-1}.
\]
Moreover, this equality holds if and only if $n=(p-1)/2$.
\end{lem}
\begin{proof}
Let $n=n_rp^r+n_{r-1}p^{r-1}+\cdots +n_0$
where $0\leq n_i<p$ for each $i$, especially $n_r\not=0$.
From Lemma \ref{factorial} and \ref{factorial2},
\[
v_p\left( \frac{(n!)^2}{(2n+1)!}\right)\geq -r-1>-\frac{2n}{p-1}
\]
for $n\geq p$ since
\[
\frac{p-1}{2}(r+1)<p^r\leq n.
\]
For $n<p$,
\[
v_p\left( \frac{(n!)^2}{(2n+1)!}\right) =\left\{
\begin{array}{ll}
0 & (0\leq n<(p-1)/2) \\
-1 & ((p-1)/2\leq n<p)
\end{array}
\right.
.
\]
Then
\[
v_p\left( \frac{(n!)^2}{(2n+1)!}\right)\geq -\frac{2n}{p-1}
\]
holds for any $n$ and the equality holds if and only if $n=(p-1)/2$.
\end{proof}
This lemma implies the next proposition.
\begin{prp}\label{divisibilityp}
For an odd prime $p$, positive integers $n_1,\cdots ,n_m$ and $l=n_1+\cdots +n_m$, then
\[
v_p\left( \frac{(n_1!)^2\cdots (n_m!)^2}{(2n_1+1)!\cdots (2n_m+1)!}\right)\geq -\frac{2l}{p-1},
\]
where the equality holds if and only if $n_i=(p-1)/2$ for each $i$.
\end{prp}
Then, by Theorem \ref{formula}, we have
\[
\epsilon_{n(p-1)/2}\equiv (-1)^{n(p+1)/2}\frac{\left( \frac{p-1}{2}!\right)^{2n}}{(p!)^n} \mod \, p^{-n+1}\bm{Z}_{(p)}.
\]
\begin{thm}\label{coeffp}
For a non-negative integer $n$,
\[
v_p(\epsilon_{n(p-1)/2})=-n.
\]
\end{thm}
Especially, $v_3(\epsilon_l)=-l$.
We also have the following estimate.
\begin{thm}\label{coeffpp}
For a non-negative integer $l<n(p-1)/2$,
\[
v_p(\epsilon_l)>-n
\]
\end{thm}
\begin{proof}
Let positive integers $i_1,\cdots ,i_m$ satisfy $i_1+\cdots +i_m=l$.
From Proposition \ref{divisibilityp},
\[
v_p\left( \frac{(i_1!)^2\cdots (i_m!)^2}{(2i_1+1)!\cdots (2i_m+1)!}\right)\geq -\frac{2l}{p-1}>-n,
\]
Therefore, by Theorem \ref{formula}, $v_p(\epsilon_l)>-n$.
\end{proof}
These results imply $d'_p(k)=(p-1)v_p(k)/2$.
\section{Further observation}
Though it suffices to know Theorem \ref{coeff2}, \ref{coeffp} and \ref{coeffpp} for our application, we see the divisibility by 5 here.
\par
For $l=2n$, by Theorem \ref{coeffp}, $v_5(\epsilon_{2n})=-n$.
Then we consider the case $l=2n+1$.
Since $v_5(\epsilon_{2n+1})\geq -n$,
\begin{align*}
\epsilon_{2n+1}\equiv (-1)^{n+1}n\cdot \frac{(2!)^{2n-2}(3!)^2}{(5!)^{n-1}7!}+(-1)^n(n+1)\cdot \frac{(2!)^{2n}(1!)^2}{(5!)^n3!}\mod \,5^{-n+1}\bm{Z}_{(5)}.
\end{align*}
The right hand side is computed as
\[
(-1)^n\frac{2^{2n}(7-2n)}{(5!)^{n-1}7!}.
\]
Then if $l\equiv 3\mod 10$, $v_5(\epsilon_l)>-[l/2]$, where $[l/2]$ represents the largest integer $\leq l/2$.
On the other hand, if $l\not \equiv 3\mod 10$, $v_5(\epsilon_l)=-[l/2]$.
\par
Actually, $\epsilon_l$ ($l=1,\cdots ,20$) is computed as follows, where the right hand sides are the prime factorizations.
\begin{align*}
\epsilon_1&=2^{-1}3^{-1}\\
\epsilon_2&=2^{-2}3^{-2}5^{-1}(-1)\\
\epsilon_3&=2^{-3}3^{-3}5^07^{-1}\\
\epsilon_4&=2^{-4}3^{-4}5^{-2}7^{-1}(-1)23\\
\epsilon_5&=2^{-5}3^{-5}5^{-2}7^{-1}11^{-1}263\\
\epsilon_6&=2^{-6}3^{-6}5^{-3}7^{-2}11^{-1}13^{-1}(-1)353\cdot 379\\
\epsilon_7&=2^{-7}3^{-7}5^{-3}7^{-2}11^{-1}13^{-1}197\cdot 797\\
\epsilon_8&=2^{-8}3^{-8}5^{-4}7^{-2}11^{-1}13^{-1}17^{-1}(-1)383\cdot 42337\\
\epsilon_9&=2^{-9}3^{-9}5^{-4}7^{-3}11^{-1}13^{-1}17^{-1}19^{-1}2689453969\\
\epsilon_{10}&=2^{-10}3^{-10}5^{-5}7^{-2}11^{-2}13^{-1}17^{-1}19^{-1}(-1)26893118531\\
\epsilon_{11}&=2^{-11}3^{-11}5^{-5}7^{-3}11^{-2}13^{-1}17^{-1}19^{-1}23^{-1}73\cdot 76722629153\\
\epsilon_{12}&=2^{-12}3^{-12}5^{-6}7^{-4}11^{-2}13^{-2}17^{-1}19^{-1}23^{-1}(-1)127\cdot 563\cdot 46721395729\\
\epsilon_{13}&=2^{-13}3^{-13}5^{-5}7^{-4}11^{-2}13^{-2}17^{-1}19^{-1}23^{-1}71\cdot 1531\cdot 20479\cdot 397849\\
\epsilon_{14}&=2^{-14}3^{-14}5^{-7}7^{-4}11^{-2}13^{-2}17^{-1}19^{-1}23^{-1}29^{-1}(-1)43\cdot 19981442744694143\\
\epsilon_{15}&=2^{-15}3^{-15}5^{-7}7^{-5}11^{-3}13^{-2}17^{-1}19^{-1}23^{-1}29^{-1}31^{-1}233\cdot 11874127314767975461\\
\epsilon_{16}&=2^{-16}3^{-16}5^{-8}7^{-5}11^{-3}13^{-2}17^{-2}19^{-1}23^{-1}29^{-1}31^{-1}(-1)319473088311274492668499\\
\epsilon_{17}&=2^{-17}3^{-17}5^{-8}7^{-5}11^{-3}13^{-2}17^{-2}19^{-1}23^{-1}29^{-1}31^{-1}103\cdot 191\cdot 11677\cdot 8295097\cdot 229156549\\
\epsilon_{18}&=2^{-18}3^{-18}5^{-9}7^{-6}11^{-3}13^{-3}17^{-2}19^{-2}23^{-1}29^{-1}31^{-1}37^{-1}(-1)811\cdot 236696258753425486925956793\\
\epsilon_{19}&=2^{-19}3^{-19}5^{-9}7^{-6}11^{-3}13^{-3}17^{-2}19^{-2}23^{-1}29^{-1}31^{-1}37^{-1}276162497983\cdot 959905866507242503\\
\epsilon_{20}&=2^{-20}3^{-20}5^{-10}7^{-6}11^{-4}13^{-3}17^{-2}19^{-2}23^{-1}29^{-1}31^{-1}37^{-1}41^{-1}(-1)269\cdot 13677071637569\cdot 225347651134721497\\
\end{align*}
\section{Applications to $A_n$-types of gauge groups}
As in \S 2, we assume there exists the following homotopy commutative diagram:
\[
\xymatrix{
S^4\vee \bm{H}P^n \ar[r]^-{k\vee i} \ar[d]_j & \bm{H}P^\infty \vee \bm{H}P^\infty \ar[r]^{\nabla} & \bm{H}P^\infty \ar[d]^{\rm localization}\\
S^4\times \bm{H}P^n \ar[rr]^-f & & \bm{H}P^\infty_{(p)} 
}
\]
where $p$ is a prime and $i$ and $j$ are the inclusions.
Let us consider the map
\[
S^4\times \bm{H}P^n\cup *\times \bm{H}P^{n+1}\stackrel{f\cup (({\rm localization})i)}{\to}\bm{H}P^\infty_{(p)}.
\]
The obstruction to extending this map over $S^4\times \bm{H}P^{n+1}$ lives in $\pi_{4n+7}(\bm{H}P^\infty_{(p)})$.
Then, from Theorem of \cite{Sel78}, the obstruction to extending the map
\[
S^4\times \bm{H}P^n\cup *\times \bm{H}P^{n+1}\stackrel{(p\times {\rm id})\cup {\rm id}}{\to}S^4\times \bm{H}P^n\cup *\times \bm{H}P^{n+1}\stackrel{f\cup (({\rm localization})i)}{\to}\bm{H}P^\infty_{(p)}.
\]
over $S^4\times \bm{H}P^{n+1}$ vanishes for an odd prime $p$.
Hence one can see $d_p(pk)>d_p(k)$ and $d_p(k)\geq v_p(k)$ inductively.
For $p=2$, from \cite{Jam57}, $d_2(4k)>d_2(k)$ and $d_2(k)\geq [v_2(k)/2]$ similarly.
Then we have
\[
v_p(k)\leq d_p(k)\leq \frac{p-1}{2}v_p(k)
\]
for an odd prime $p$ and
\[
\left[ \frac{v_2(k)}{2}\right]\leq d_2(k)\leq v_2(k)
\]
from previous two sections.
Especially, $d_3(k)=v_3(k)$.
\par
Now we give the lower bound of the number of $A_n$-types of gauge groups of principal $\SU(2)$-bundle over $S^4$.
As stated in \S 2, if $\mathcal{G}(P_k)$ and $\mathcal{G}(P_{k'})$ are $A_n$-equivalent, then $\min \{n,d_p(k)\}=\min \{n,d_p(k')\}$ for any prime $p$.
If $p$ is an odd prime, then
\[
\# \{ \,\min \{ n,d_p(k) \} \, |\, k\in \bm{Z} \, \}\geq \left[\frac{2n}{p-1}+1 \right]
\]
since $0=d_p(1)<d_p(p)<d_p(p^2)<\cdots <d_p(p^{[2n/(p-1)]})\leq n$.
If $p=2$, then
\[
\# \{ \,\min \{ n,d_2(k) \} \, |\, k\in \bm{Z} \, \}\geq \left[\frac{n}{2}+1 \right]
\]
since $0=d_2(1)<d_2(4)<d_2(16)<\cdots <d_2(4^{[n/2]})\leq n$.
\begin{thm}
The number of $A_n$-types of gauge groups of principal $\SU(2)$-bundles over $S^4$ is greater than
\[
\left[\frac{n}{2}+1 \right]\prod_{p:{\rm odd\, prime}}\left[\frac{2n}{p-1}+1 \right].
\]
\end{thm}
We can express the logarithm of this as follows:
\begin{align*}
\log\left( \left[\frac{n}{2}+1 \right]\prod_{p:{\rm odd\, prime}}\left[\frac{2n}{p-1}+1 \right] \right)
=&\log \left[\frac{n}{2}+1 \right]+\sum_{p:{\rm odd\, prime}}\log \left[\frac{2n}{p-1}+1 \right]\\
=&\log \left[\frac{n}{2}+1 \right]+\sum_{r=2}^{n+1}\left(\pi\left(\frac{2n}{r-1}+1\right)-1\right)(\log r-\log (r-1))\\
=&\log \left[\frac{n}{2}+1 \right]+\sum_{r=1}^n\pi \left( \frac{2n}{r}+1\right)\log \left( 1+\frac{1}{r}\right)-\log (n+1),
\end{align*}
where $\pi$ is the prime counting function.
The second equality is seen by
\[
\# \left\{ p:{\rm an\, odd\, prime}\, \left| \, \left[\frac{2n}{p-1}+1 \right]\geq r\right. \right\}=\pi\left(\frac{2n}{r-1}+1\right)-1.
\]

\end{document}